%% file: main.tex
\documentclass[]{interact}

\usepackage{epstopdf}
\usepackage[caption=false]{subfig}
\usepackage{lmodern} 

\usepackage[numbers,sort&compress]{natbib}
\bibpunct[, ]{[}{]}{,}{n}{,}{,}

\theoremstyle{plain}
\newtheorem{theorem}{Theorem}[section]
\newtheorem{lemma}[theorem]{Lemma}

\theoremstyle{definition}
\newtheorem{definition}[theorem]{Definition}

\theoremstyle{remark}
\newtheorem{remark}{Remark}

\newtheorem{assumption}{Assumption}[section]

\usepackage{cleveref}
\crefname{assumption}{Assumption}{Assumptions}

\input{lvc_style}

\usepackage{tikz}

\begin{document}


\title{Parabolic PDEs on a fixed domain with evolving subdomains: function spaces and well-posedness}

\author{
\name{Van Chien Le\textsuperscript{a}\thanks{CONTACT V.~C. Le. Email: vanchien.le@ugent.be} and Karel Van Bockstal\textsuperscript{b}}
\affil{\textsuperscript{a}IDLab, Department of Information Technology, Ghent University -- imec, Ghent, Belgium; \textsuperscript{b}Ghent Analysis \& PDE center, Department of Mathematics: Analysis, Logic and Discrete Mathematics, Ghent University, Ghent, Belgium}
}

\maketitle

\begin{abstract}
    This paper develops the necessary ingredients for the variational approach of initial boundary-value problems of parabolic partial differential equations on a fixed spatial domain containing evolving subdomains. In particular, we introduce function spaces for the variational solution that extend standard Sobolev--Bochner spaces to account for a coefficient associated with the time derivative that may be discontinuous across the evolving interface. We further show the density of smooth functions in these spaces by extending the mollification technique and the Reynolds transport theorem, and establish the corresponding ``embedding'' theory and an integration by parts formula. Finally, we prove the well-posedness of the space-time variational formulation in the natural setting using the Banach--Ne\v{c}as--Babu\v{s}ka theorem.
\end{abstract}


\section{Introduction}

Let $\Om \sst \R^d$, with $d = 1, 2, 3,$ be a fixed, open, and bounded domain with Lipschitz boundary $\pa\Om$, and the time interval $J := (0, T) \sst \R$, with $0 < T < \infty$. The spatial domain $\Om$ consists of two nonempty open Lipschitz subdomains $\Om_1(t)$ and $\Om_2(t)$, which are separated by an interface $\Gm(t)$ evolving in time (see \Cref{fig:combined}). In other words, 
\[
    \Om = \Om_1(t) \cup \Om_2(t) \cup \Gm(t), \qqq \Omega_1(t) \cap \Omega_2(t) = \emptyset, \qqq \Gm(t) = \pa\Om_1(t) \cap \pa\Om_2(t), 
\]
for all $t \in \ovl{J}$. We denote the space-time domain $Q := \Om \times J \sst \R^{d+1}$ and its subdomains
\[
    Q_1 = \brac{(\xb, t) \in Q: \xb \in \Om_1(t), t \in J}, \qqqq Q_2 = Q \setminus \ovl{Q_1}.
\]
The evolution of the interface and subdomains over time is governed by a prescribed velocity vector $\vb$. More specifically, we assume that there exists a sufficiently smooth mapping $\Phi : \ovl{\Om} \times \ovl{J} \to \ovl{\Om}$ such that $\Phi(\cdot, 0) = \operatorname{Id}$ and $\Phi(\cdot, t)$ is a diffeomorphism that maps from $\ovl{\Om_1(0)}$ and $\ovl{\Om_2(0)}$ onto $\ovl{\Om_1(t)}$ and $\ovl{\Om_2(t)}$, respectively, for all $t \in J$. The velocity vector $\vb : \ovl{Q} \to \R^d$ is defined by
\[
    \vb(\xb, t) = \dot{\Phi}\paren{\left[\Phi(\cdot, t)\right]^{-1}(\xb), t}, \qqqqq (\xb, t) \in \ovl{Q},
\]
where $\dot{\Phi}$ stands for the total derivative of $\Phi$ with respective to $t$, and $\left[\Phi(\cdot, t)\right]^{-1}$ is the inverse of $\Phi(\cdot, t)$. Throughout the paper, we assume that $\vb\in \Cs^1\paren{\ovl{Q}}^d$.

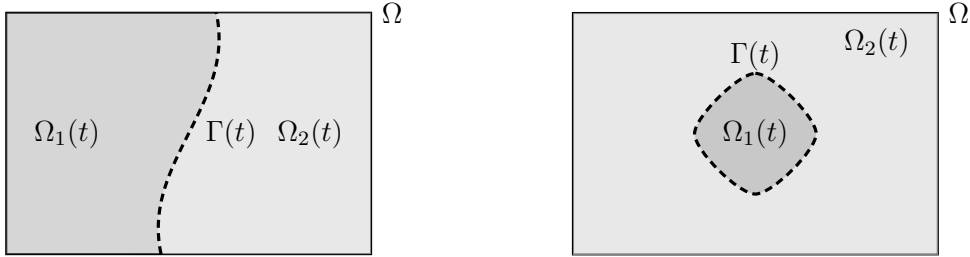
\begin{figure}[ht]
    \centering
    \begin{minipage}{0.48\textwidth}
        \centering
        \begin{tikzpicture}[scale=0.8]
            \draw[black, thick] (-3,-2) rectangle (3,2) node[right] {$\Omega$};
            
            \fill[gray!80, opacity=0.4] (-3,-2) -- (-3,2) -- plot[domain=2:-2, samples=50] ({0.5*sin(deg(\x))}, {\x}) -- cycle;
            
            \fill[gray!30, opacity=0.6] plot[domain=-2:2, samples=50] ({0.5*sin(deg(\x))}, {\x}) -- (3,2) -- (3,-2) -- cycle;
            
            \draw[black, very thick, densely dashed, domain=-2:2, samples=50] plot ({0.5*sin(deg(\x))}, {\x});
            
            \node[black] at (0.7, 0) {$\Gamma(t)$};
            \node[black] at (-2,0) {$\Omega_1(t)$};
            \node[black] at (2,0) {$\Omega_2(t)$};
        \end{tikzpicture}
    \end{minipage}
    \hfill
    \begin{minipage}{0.48\textwidth}
        \centering
        \begin{tikzpicture}[scale=0.8]
            \draw[black, thick] (-3,-2) rectangle (3,2) node[right] {$\Omega$};
            \fill[gray!30, opacity=0.6] (-3,-2) rectangle (3,2);
            \fill[gray!80, opacity=0.4] plot[smooth cycle] coordinates {(-1,0) (0,1) (1,0) (0,-1)};
            \draw[black, very thick, densely dashed] plot[smooth cycle] coordinates {(-1,0) (0,1) (1,0) (0,-1)};
            
            \node[black] at (0,1.3) {$\Gamma(t)$};
            \node[black] at (0,0) {$\Omega_1(t)$};
            \node[black] at (2,1.5) {$\Omega_2(t)$};
        \end{tikzpicture}
    \end{minipage}
    \caption{Illustrations of a fixed domain $\Omega$ comprising two subdomains $\Omega_1(t)$ and $\Omega_2(t)$, separated by an evolving interface $\Gamma(t)$. Left: The interface intersects the boundary of $\Omega$. Right: The closed interface divides the domain into two subdomains, with one entirely enclosed within the other.}
    \label{fig:combined}
\end{figure}

The two subdomains $\Om_1(t)$ and $\Om_2(t)$ are filled by materials that are characterized by a non-negative coefficient $\alpha$ satisfying
\[
    \alpha \in \Cs^1\paren{\ovl{Q_1}} \cap \Cs^1\paren{\ovl{Q_2}}.
\]
To avoid the trivial case, we assume $\alpha$ not to be zero on both subdomains, i.e., $\supp(\alpha) \not\equiv \emptyset$. For all $t \in \ovl{J}$, let
\[
    \emptyset \not\equiv \Sm(t) := \supp(\alpha(\cdot, t)) \setminus (\Gm(t) \cup \pa\Om) \sst \Om.
\]
On $\Sm(t)$, the coefficient $\alpha(\cdot, t)$ is assumed to be strictly positive, i.e.,
\[
    \alpha(\xb, t) \ge \alpha_0 > 0, \qqqq \fa \xb \in \Sigma(t), \ \fa t \in \ovl{J}.
\]
As a consequence, the subdomain $\Sm(t)$ is either $\Om_1(t)$ or $\Om_2(t)$, or their union. Please keep in mind that we do not assume continuity of $\alpha$ across the interface $\Gm(t)$.

We now formulate the abstract linear (elliptic-)parabolic equation that will be the focus of our analysis. Let $\Vs \emb \Hs^1(\Om)$ be a Hilbert space such that the embedding $\Vs \emb \Ls^2(\Om)$ is dense. The dual space of $\Vs$ is denoted by $\Vs^\prime$. Hence, $\Vs \emb \Ls^2(\Om) \emb \Vs^\prime$ forms a Gelfand triple. Given $g_1 \in \Ls^2(J, \Vs^\prime)$ and a family of bounded linear operators $A(t) \in \LL(\Vs, \Vs^\prime)$ for almost all (a.a.) $t \in J$, we consider the evolution equation
\begin{equation}
\label{eq:pde}
    (\alpha \pa_t u)(t) + A(t) u(t) = g_1(t) \qqqq \text{in} \q \Vs^\prime,
\end{equation}
for a.a. $t \in J$. \Cref{eq:pde} can model a broad class of physical phenomena, including phase transition (e.g., Cahn--Hilliard and Allen--Cahn equations) \cite{GXC2025,SC2025,WWC2025}, mass transfer (e.g., Stokes and Oseen flows) \cite{Gross2011,Pruess2016,VR2018,MTZ2023}, heat transfer \cite{Slodicka2021,LSV2024}, and electromagnetics (eddy-current problems) \cite{LSV2021a,LSV2021b,LSV2022c}.

With Galerkin numerical methods in mind, our goal is to derive a variational formulation for \eqref{eq:pde} and establish its well-posedness under suitable assumptions. This is a well-established result for the case $\alpha = 1$ or when $\alpha$ is sufficiently smooth and strictly positive, see, e.g., \cite{Dautray1992,SS2009}. In such settings, the weak solution $u$ to \eqref{eq:pde} (together with an appropriate initial condition) is uniquely determined in the Sobolev--Bochner space
\[
    \Ws^{1, 2, 2}(J, \Vs, \Vs^\prime) := \brac{u \in \Ls^2(J, \Vs) : \pa_t u \in \Ls^2(J, \Vs^\prime)}.
\]
Numerous studies have demonstrated that this result remains valid in the presence of evolving subdomains and interfaces. We refer the reader to \cite{FR2017,LRS2020,Guo2021,BDV2023,NLH+2025} for further details.

When $\alpha$ is discontinuous across the evolving interface, the main challenge lies in rigorously defining $\alpha \pa_t u$ as element in a Bochner space taking values in the dual space $\Vs^\prime$, such that the corresponding Sobolev--Bochner space for the solution possesses the necessary properties for the variational analysis to be applicable. These properties include an integration by parts formula and a continuous embedding theory by which the initial value of $u$ at $t = 0$ is well-defined. In \cite{VR2018}, the authors introduced an analogue of the space $\Ws^{1, 2, 2}(J, \Vs, \Vs^\prime)$, in which the weak time derivative $\pa_t u$ is replaced by the weak material derivative $\alpha \dot{u} := \alpha \pa_t u + \alpha \vb \cdot \nabla u$. The definition of $\alpha \dot{u}$ in the distributional sense naturally generalizes that of $\pa_t u$ by taking into account the jump of $\alpha$ across the evolving interface. Nonetheless, the authors encountered a fundamental difficulty: the lack of density of smooth functions in this space, which prevented them from developing the necessary tools for the variational analysis, see \cite[Remark~2.4]{VR2018}. This issue was then addressed in \cite{LSV2024}, where the authors defined $\alpha \pa_t u \in \Ls^1(J, \Vs^\prime)$, with $\Vs = \Hs^1(\Om)$, by requiring the absolute continuity in time of the $\Ls^2(\Om)$-inner product between $\alpha u$ and $w$, for any $w \in \Hs^1(\Om)$. However, this is not a natural approach to define $\alpha \pa_t u$.

In this paper, we define $\alpha \pa_t u$ as a distribution taking values in $\Vs^\prime$ by generalizing the notion of the weak time derivative $\pa_t u$ and incorporating the evolution of the subdomains and interface via the Reynolds transport theorem. Under suitable conditions, this definition is shown to be equivalent to $\alpha \dot{u}$ introduced in \cite{VR2018}, and to the definition of $\alpha \pa_t u$ proposed in \cite{LSV2024}. To support this construction, we establish the density of smooth functions in the associated Sobolev--Bochner space for the weak solution by extending the classical mollification technique. Building on this density property, we develop the corresponding ``embedding'' theory and an integration by parts formula, which are essential in deriving the variational formulation of \eqref{eq:pde}. The well-posedness of the variational problem is then proven using the standard Banach--Ne\v{c}as--Babu\v{s}ka framework.

We also mention an alternative approach to solving \eqref{eq:pde}, which interprets the equation as being defined on two evolving subdomains $\Om_1(t)$ and $\Om_2(t)$, with transmission or jump conditions imposed at the shared interface $\Gm(t)$. Function spaces and finite element frameworks for parabolic partial differential equations (PDEs) on moving domains have been developed in the literature, which rely on the pushforward and pullback mappings $\Phi(\cdot, t)$ and $\Phi(\cdot, t)^{-1}$ to define the evolving function spaces. The interested reader is referred to \cite{AES2015,ER2021,ACD+2023,DGH2023} for more details. However, these evolving function spaces do not preserve the tensor product structure of the space-time domain $Q$. In this work, instead, we construct function spaces that maintain the tensor product structure. This approach facilitates the use of tensor-structured numerical methods for solving \eqref{eq:pde}, such as classical time-stepping finite element schemes or structured space-time finite element methods.

The paper is organized as follows. In the next section, we introduce a generalization of the standard Sobolev--Bochner spaces for the weak solution of \eqref{eq:pde}, along with their fundamental properties. In \Cref{sec:varf}, we derive a space-time variational formulation for \eqref{eq:pde}, and prove its well-posedness under appropriate assumptions. Concluding remarks and potential directions for future research are presented in \Cref{sec:conclusions}.

\section{Function spaces}

Given a Banach space $\Xs$ with its dual $\Xs^\prime$, the norm on $\Xs$ and the duality pairing on $\Xs^\prime \times \Xs$ are denoted by $\norm{\cdot}_{\Xs}$ and $\inprod{\cdot, \cdot}$, respectively. The inner product in $\Ls^2(\Om)$ is denoted by $\paren{\cdot, \cdot}$. For the sake of brevity, the dependence of functions on time and space variables is omitted if it does not cause any ambiguity. In further analysis, the following positive constants are frequently used 
\[
    C_{\vb} := \norm{\vb}_{\Cs^1\paren{\ovl{Q}}^d}, \qqqq \ C_{\alpha} := \max\brac{\norm{\alpha}_{\Cs^1\paren{\ovl{Q_1}}}, \norm{\alpha}_{\Cs^1\paren{\ovl{Q_2}}}}.
\]
The norm of the continuous embedding $\Vs \emb \Hs^1(\Om)$ is denoted by $C_{\text{emb}}$, i.e.,
\[
    \norm{u}_{\Hs^1(\Om)} \le C_{\text{emb}} \norm{u}_{\Vs}, 
    \qqqqqq \forall u \in \Vs.
\]

We recall here the Reynolds transport theorem, which is useful for handling time derivatives of integrals over evolving domains. Let 
\[
    \wt{\Om}(t) := \Om_1(t) \cup \Om_2(t), \qqqq \wt{Q} := Q_1 \cup Q_2.
\]
Given $f \in \Ws^{1,1}(\wt{Q})$, the Reynolds transport theorem \cite[p.~78]{Gurtin1981} states that 
\begin{equation}
    \label{eq:RTT}
    \dfrac{\di}{\dt} \int\limits_{\wt{\Om}(t)} f \dx = \int\limits_{\wt{\Om}(t)} \pa_t f \dx + \int\limits_{\pa\wt{\Om}(t)} f \vb \cdot \nv \ds = \int\limits_{\wt{\Om}(t)} \pa_t f \dx + \int\limits_{\wt{\Om}(t)} \dive{(f \vb)} \dx,
\end{equation}
for a.a. $t \in J$, where $\nv$ is the outward unit normal of $\pa\wt{\Om}(t)$. 

\begin{remark}
    We note that $\wt{\Om}(t)$ and $\Om$ differ only on a set of measure zero for all $t \in \overline{J}$, so do $\wt{Q}$ and $Q$. However, the material coefficient $\alpha \in \Cs^1(\wt{Q})$, while $\alpha \in \Ls^{\infty}(Q)$ only. For this reason, on some occasions in further analysis, integrals over $\Om$ will be (implicitly) reduced to those over $\wt{\Om}(t)$ to enable the application of the Reynolds transport theorem \eqref{eq:RTT}. Throughout the paper, we also adopt the convention that $\pa_t \alpha \in \Ls^\infty(Q)$ and $\nabla \alpha \in \Ls^\infty(Q)^d$, which should be understood as $\pa_t \alpha \in \Cs\paren{\ovl{Q_1}} \cap \Cs\paren{\ovl{Q_2}}$ and $\nabla \alpha \in \Cs\paren{\ovl{Q_1}}^d \cap \Cs\paren{\ovl{Q_2}}^d$.
\end{remark}

Next, we define $\alpha \pa_t u$ as a distribution taking values in $\Vs^\prime$. 

\begin{definition}
\label{def:derivative}
    Given $u \in \Ls^1(J, \Vs)$. Then, $\alpha \pa_t u \in \LL\paren{\DD(J), \Vs^\prime}$ is defined by
    \begin{equation}
    \label{eq:def}
        \inprod{(\alpha \pa_t u)(\phi), w} = - \int\limits_0^T \int\limits_\Om \paren{\alpha u w \pa_t \phi + u \phi w \pa_t \alpha + \phi \dive{\paren{\alpha u w \vb}}} \dx \dt,
    \end{equation}
    for any $\phi \in \DD(J)$ and $w \in \Vs$, where $\DD(J)$ stands for the set of infinitely differentiable functions with a compact support in $J$.
\end{definition}

The condition $\Vs \subseteq \Hs^1(\Om)$ guarantees that the integrand $\dive{(\alpha u w \vb)}$ is $\Ls^1(Q)$-integrable. One can easily see that if $u \in \Ls^1(J, \Vs)$ satisfies $\pa_t u \in \Ls^1(J, \Ls^2(\Om))$, then $\alpha \pa_t u$ defined by \eqref{eq:def} belongs to $\Ls^1(J, \Ls^2(\Om))$, and it coincides with the product between $\alpha $ and $\pa_t u$. Indeed, applying integration by parts and the Reynolds transport theorem \eqref{eq:RTT} on the right-hand side of \eqref{eq:def} leads to
\[
    \inprod{(\alpha \pa_t u)(\phi), w} = \int\limits_0^T \phi(t) \paren{\alpha(t) \pa_t u(t), w} \dt,
\]
which holds true for any $\phi \in \DD(J)$ and $w \in \Vs$. Then, invoking the density of $\Vs$ in $\Ls^2(\Om)$, we end up with the above conclusion.

Let $p \in (1, \infty)$ and $p^\prime := p/(p-1)$ its conjugate exponent. The following space is analogous to the Sobolev--Bochner space $\Ws^{1, p, p^\prime}(J, \Vs, \Vs^\prime)$ (see, e.g., \cite[Chapter~7]{Roubicek2005})
\[
    \Ws^{1, p, p^\prime}_{\alpha}(J, \Vs, \Vs^\prime) := \brac{u \in \Ls^p(J, \Vs) : \alpha \pa_t u \in \Ls^{p^\prime}(J, \Vs^\prime)},
\]
equipped with the graph norm
\[
    \norm{u}_{\Ws^{1, p, p^\prime}_{\alpha}(J, \Vs, \Vs^\prime)} := \norm{u}_{\Ls^p(J, \Vs)} + \norm{\alpha \pa_t u}_{\Ls^{p^\prime}(J, \Vs^\prime)}.
\]

Let $\DD(J, \Vs)$ be the set of infinitely differentiable functions with a compact support in $J$ and taking values in $\Vs$. The following lemma demonstrates the density of $\DD(J, \Vs)$ in the spaces $\Ws^{1, p, p^\prime}_{\alpha}(J, \Vs, \Vs^\prime)$. The proof employs the classical mollification technique, see, e.g., \cite[Lemma~7.2]{Roubicek2005} and \cite[Lemma~3.1]{LSV2024}.

\begin{lemma}
\label{lem:density}
    If $p \in [2, \infty)$, then $\DD(J, \Vs)$ is densely contained in $\Ws^{1, p, p^\prime}_{\alpha}(J, \Vs, \Vs^\prime)$.
\end{lemma}

\begin{proof}
    Let $u \in \Ws^{1, p, p^\prime}_{\alpha}(J, \Vs, \Vs^\prime)$, with $p \in [2, \infty)$. For a sufficiently small $\veps > 0$, we set
    \[
        u_\veps(t) := \int\limits_0^T \rho_\veps\paren{t + \xi_\veps(t) - s} u(s) \ds,
    \]
    where $\rho_\veps : \R \to \R$ is the mollifier defined by
    \[
        \rho_\veps(t) :=
        \begin{cases}
            c\veps^{-1} \exp\paren{t^2/\paren{t^2-\veps^2}} \qq &\text{for } \abs{t} < \veps, \\
            0  & \text{elsewhere},
        \end{cases}
    \]
    with $c$ the constant so that $\int_\R \rho_{1}(t) \dt = 1$. The function $\xi_{\veps} : J \to \paren{-\veps, \veps}$ given by
    \[
        \xi_\veps(t) = \veps \dfrac{T - 2t}{T}
    \]
    slightly shifts the kernel in the convolution integral so that only values of $u$ inside $[0, T]$ are taken into account. It is obvious that $u_\veps \in \DD(J, \Vs)$. In addition, $u_\veps$ converges to $u$ in $\Ls^p(J, \Vs)$ when $\veps \to 0$, following the proof of \cite[Lemma~7.2]{Roubicek2005}. In the rest of the proof, we shall show that $\alpha \pa_t u_\veps$ also converges to $\alpha \pa_t u$ in $\Ls^{p^\prime}(J, \Vs^\prime)$ when $\veps \to 0$. More specifically, we need to prove that
    \begin{equation}
    \label{eq:conv_dualV}
        \lim_{\veps \to 0} \norm{\alpha \pa_t u_\veps - \alpha \pa_t u}_{\Ls^{p^\prime}(J, \Vs^\prime)} = \lim_{\veps \to 0} \, \sup_{z \in \Ls^p(J, \Vs) \setminus \brac{0}} \dfrac{\abs{\inprod{\alpha \pa_t u_\veps - \alpha \pa_t u, z}}}{\norm{z}_{\Ls^p(J, \Vs)}} = 0.
    \end{equation}
    For all $z \in \Ls^p(J, \Vs)$, we can write
    \[
        \inprod{\alpha \pa_t u, z} = \int\limits_0^T \inprod{(\alpha \pa_t u)(t), z(t)} \dt
        = \int\limits_0^T \int\limits_0^T \rho_\veps\paren{t + \xi_\veps(t) - s} \inprod{(\alpha \pa_t u)(t), z(t)} \ds \dt.
    \]
    The duality pairing between $\alpha \pa_t u_\veps$ and $z$ reads
    \begin{align*}
        \inprod{\alpha \pa_t u_\veps, z}
        & = \int\limits_0^T \paren{\alpha(t) \pa_t u_\veps(t), z(t)} \dt \\
        & = C_\veps \int\limits_0^T \int\limits_\Om \int\limits_0^T \alpha(t) \, \rho^\prime_\veps\paren{t + \xi_\veps(t) - s} u(s) \, z(t) \ds \dx \dt, 
    \end{align*}
    where $C_\veps := (T-2\veps)/T$. Note that for a fixed $t \in J, z(t) \in \Vs$ and $\phi \in \DD(J)$, with $\phi(s) := \rho_\veps\paren{t + \xi_\veps(t) - s}$. Hence, \Cref{def:derivative} can be invoked to obtain that
    \begin{align*}
    \inprod{\alpha \pa_t u_\veps, z} & = C_\veps \int\limits_0^T \int\limits_0^T  \int\limits_\Om \rho^\prime_\veps\paren{t + \xi_\veps(t) - s} \alpha(s) \, u(s) \, z(t) \dx \ds \dt \\
        & \q\, + C_\veps \int\limits_0^T \int\limits_0^T \int\limits_\Om \rho^\prime_\veps\paren{t + \xi_\veps(t) - s} \paren{\alpha(t) - \alpha(s)} u(s) \, z(t) \dx \ds \dt \\
        & = C_\veps \int\limits_0^T \int\limits_0^T \rho_\veps\paren{t + \xi_\veps(t) - s} \inprod{(\alpha \pa_s u)(s), z(t)} \ds \dt \\
        & \q\, + C_\veps \int\limits_0^T \int\limits_0^T \int\limits_{\Om} \rho_\veps\paren{t + \xi_\veps(t) - s} \pa_s \alpha(s) \, u(s) \, z(t) \dx \ds \dt \\
        & \q\, + C_\veps \int\limits_0^T \int\limits_0^T \int\limits_{\Om} \rho_\veps\paren{t + \xi_\veps(t) - s} \dive{\paren{\alpha(s) \, u(s) \, z(t) \, \vb(s)}} \dx \ds \dt \\
        & \q\, + C_\veps \int\limits_0^T \int\limits_0^T \int\limits_\Om \rho^\prime_\veps\paren{t + \xi_\veps(t) - s} \paren{\alpha(t) - \alpha(s)} u(s) \, z(t) \dx \ds \dt \\
        & =: C_\veps \paren{I_1 + I_2 + I_3 + I_4}.
    \end{align*} 
    Employing the Reynolds transport theorem \eqref{eq:RTT}, $I_4$ can be rewritten as
    \begin{align*}
        I_4 & = \int\limits_0^T \int\limits_0^T \rho_\veps^\prime\paren{t + \xi_\veps(t) - s} \paren{\,\, \int\limits_{\wt{\Om}(t)} \alpha(t) \, u(s) \, z(t) \dx - \int\limits_{\wt{\Om}(s)} \alpha(s) \, u(s) \, z(t) \dx} \ds \dt \\
        & = \int\limits_0^T \int\limits_0^T \rho_\veps^\prime\paren{t + \xi_\veps(t) - s} \int\limits_s^t \dfrac{\di}{\di \tau} \int\limits_{\wt{\Om}(\tau)} \alpha(\tau) \, u(s) \, z(t) \dx \di \tau \ds \dt \\
        & = \int\limits_0^T \int\limits_0^T \rho_\veps^\prime\paren{t + \xi_\veps(t) - s} \int\limits_s^t \int\limits_{\Om} \pa_\tau \alpha(\tau) \, u(s) \, z(t) \dx \di \tau \ds \dt \\
        & \q\, + \int\limits_0^T \int\limits_0^T \rho_\veps^\prime\paren{t + \xi_\veps(t) - s} \int\limits_s^t \int\limits_{\Om} \dive{\paren{\alpha(\tau) \, u(s) \, z(t) \, \vb(\tau)}} \dx \di \tau \ds \dt \\
        & =: I_4^a + I_4^b.
    \end{align*}
    The numerator in \eqref{eq:conv_dualV} then becomes
    \[
        \inprod{\alpha \pa_t u_\veps - \alpha \pa_t u, z}
        = C_\veps \paren{I_1 - \inprod{\alpha \pa_t u, z}} + C_\veps \paren{I_2 + I_4^a} + C_\veps \paren{I_3 + I_4^b} - \dfrac{2\veps}{T} \inprod{\alpha \pa_t u, z}.
    \]
    In the following, we estimate the bound for each term on the right-hand side. For the first term, the use of H\"{o}lder's inequality gives that
    \begin{align*}
        & \abs{I_1 - \inprod{\alpha \pa_t u, z}}^{p^\prime} \\
        & \q = \abs{\, \int\limits_0^T \int\limits_0^T \rho_\veps\paren{t + \xi_\veps(t) - s} \inprod{(\alpha \pa_t u)(t) - (\alpha \pa_s u)(s), z(t)} \ds \dt}^{p^\prime} \\
        & \q \le \paren{\, \int\limits_0^T \norm{z(t)}_{\Vs}\int\limits_0^T \rho_\veps\paren{t + \xi_\veps(t) - s} \norm{(\alpha \pa_t u)(t) - (\alpha \pa_s u)(s)}_{\Vs^\prime}  \ds \dt}^{p^\prime} \\
        & \q \le \norm{z}^{p^\prime}_{\Ls^p(J, \Vs)} \int\limits_0^T \paren{\,\, \int\limits_{t + \xi_\veps(t) - \veps}^{t + \xi_\veps(t) + \veps} \rho_\veps\paren{t + \xi_\veps(t) - s} \norm{(\alpha \pa_t u)(t) - (\alpha \pa_s u)(s)}_{\Vs^\prime} \ds}^{p^\prime} \dt \\
        & \q \le \norm{z}^{p^\prime}_{\Ls^p(J, \Vs)} \paren{\,\, \int\limits_{-\veps}^\veps \rho_\veps(h)^p \di h}^{p^\prime/p} \int\limits_0^T \int\limits_{\xi_\veps(t) - \veps}^{\xi_\veps(t) + \veps} \norm{(\alpha \pa_t u)(t + h) - (\alpha \pa_t u)(t)}_{\Vs^\prime}^{p^\prime} \di h \dt.
    \end{align*}
    Using the following estimate
       \[
            \int\limits_{-\veps}^\veps \rho_\veps(t)^p \dt = 
            \dfrac{c^p}{\veps^{p}} \int\limits_{-\veps}^\veps \exp\paren{\frac{-p t^2}{\veps^2-t^2}}\dt \le 2 \dfrac{c^p}{\veps^{p-1}} =: \frac{C_1^p}{(2\veps)^{p-1}}, 
        \]
    we arrive at 
    \begin{align*}
        \abs{I_1 - \inprod{\alpha \pa_t u, z}}^{p^\prime} & \le \dfrac{C_1^{p^\prime}}{2\veps} \norm{z}^{p^\prime}_{\Ls^p(J, \Vs)} \int\limits_0^T \int\limits_{\xi_\veps(t) - \veps}^{\xi_\veps(t) + \veps} \norm{(\alpha \pa_t u)(t + h) - (\alpha \pa_t u)(t)}_{\Vs^\prime}^{p^\prime} \di h \dt \\
         & \le C_1^{p^\prime} \norm{z}^{p^\prime}_{\Ls^p(J, \Vs)} \sup_{\abs{h} \leq 2 \veps} \int\limits_0^T \norm{(\alpha \pa_t u)(t + h) - (\alpha \pa_t u)(t)}_{\Vs^\prime}^{p^\prime} \dt.
    \end{align*}
    Next, integrating by parts with respect to $s$, we rewrite $I_3$ as
    \begin{align*}
        I_3 & = \int\limits_0^T \int\limits_0^T \rho_\veps\paren{t + \xi_\veps(t) - s} \int\limits_{\Om} \dive{\paren{\alpha(s) \, u(s) \, z(t) \, \vb(s)}} \dx \ds \dt \\
        & = - \int\limits_0^T \int\limits_0^T \rho_\veps\paren{t + \xi_\veps(t) - s} \dfrac{\di}{\ds} \int\limits_s^t \int\limits_{\Om} \dive{(\alpha(\tau) \, u(\tau) \, z(t) \, \vb(\tau))} \dx \di\tau \ds \dt \\
        & = - \int\limits_0^T \int\limits_0^T \rho^\prime_\veps\paren{t + \xi_\veps(t) - s} \int\limits_s^t \int\limits_{\Om} \dive{(\alpha(\tau) \, u(\tau) \, z(t) \, \vb(\tau))} \dx \di\tau \ds \dt,
    \end{align*}
    which allows to deduce the following
    \begin{align*} 
        & \abs{I_3 + I_4^b}^{p^\prime} \\
        & \,\,\, = \abs{\, \int\limits_0^T \int\limits_0^T \rho_\veps^\prime\paren{t + \xi_\veps(t) - s} \int\limits_s^t \int\limits_{\Om} \dive{\left[\alpha(\tau) \paren{u(s) - u(\tau)} z(t) \, \vb(\tau)\right]}
        \dx \di \tau \ds \dt}^{p^\prime} \\
        & \,\,\, \le C^{p^\prime}_{\alpha} C^{p^\prime}_{\vb} \paren{\, \int\limits_0^T \norm{z(t)}_{\Hs^1(\Om)} \int\limits_0^T \abs{\rho_\veps^\prime\paren{t + \xi_\veps(t) - s}} \abs{\int\limits_s^t \norm{u(s) - u(\tau)}_{\Hs^1(\Om)} \di\tau} \ds \dt}^{p^\prime} \\
        & \,\,\, \le C_2 \norm{z}^{p^\prime}_{\Ls^p(J, \Vs)} \int\limits_0^T \paren{\, \int\limits_{\xi_\veps(t) - \veps}^{\xi_\veps(t) + \veps} \abs{\rho_\veps^\prime\paren{h - \xi_\veps(t)}} \abs{\int\limits_0^h \norm{u(t + h) - u(t + \eta)}_{\Vs} \di\eta} \di h}^{p^\prime} \dt \\
        & \,\,\, \le C_2 \norm{z}^{p^\prime}_{\Ls^p(J, \Vs)} \abs{\, \int\limits_{-\veps}^\veps \abs{\rho_\veps^\prime(h)}^p \di h}^{p^\prime/p} \int\limits_0^T \int\limits_{\xi_\veps(t) - \veps}^{\xi_\veps(t) + \veps} \, \abs{\int\limits_0^h \norm{u(t + h) - u(t + \eta)}_{\Vs} \di\eta}^{p^\prime} \di h \dt,
    \end{align*}
    where $ C_2 := C^{p^\prime}_{\alpha} C^{p^\prime}_{\vb} C_{\text{emb}}^{2p^\prime}$. Note that 
    \[
            \int\limits_{-\veps}^\veps \abs{\rho^\prime_\veps(t)}^p \dt = 
            (2c\veps)^p \int\limits_{-\veps}^\veps \frac{\abs{t}^p\exp\paren{p t^2/\paren{t^2-\veps^2}}}{\paren{\veps^2 - t^2}^{2p}} \dt  
            = \dfrac{2^{p+1} c^p C_p} {\veps^{2p-1}} = \dfrac{8^{p} c^p C_p} {(2\veps)^{2p-1}},
        \]
    where 
        \[
            C_p := \int\limits_0^1 \frac{t^p\exp\paren{p t^2/\paren{t^2-1}}}{\paren{1 - t^2}^{2p}} \dt < \infty.
        \] 
    Taking into account this identity and the continuous embedding $\Ls^p(J, \Vs) \emb \Ls^{p^\prime}(J, \Vs)$, which is valid for $p \in [2, \infty)$ and has the norm $T^{(p-2)/p}$, we obtain
    \begin{align*}
        & \abs{I_3 + I_4^b}^{p^\prime} \\
            & \, \le \dfrac{C_3}{(2\veps)^{(2p - 1)/(p-1)}} \norm{z}^{p^\prime}_{\Ls^p(J, \Vs)} \int\limits_0^T \int\limits_{\xi_\veps(t) - \veps}^{\xi_\veps(t) + \veps} \abs{h}^{p^\prime/p} \int\limits_{\xi_\veps(t) - \veps}^{\xi_\veps(t) + \veps} \norm{u(t + h) - u(t + \eta)}^{p^\prime}_{\Vs} \di\eta \di h  \dt \\
            & \, \le \dfrac{C_3}{(2\veps)^2} \norm{z}^{p^\prime}_{\Ls^p(J, \Vs)} \int\limits_0^T \int\limits_{\xi_\veps(t) - \veps}^{\xi_\veps(t) + \veps} \,\, \int\limits_{\xi_\veps(t) - \veps}^{\xi_\veps(t) + \veps} \norm{u(t + h) - u(t + \eta)}^{p^\prime}_{\Vs} \di\eta \di h \dt \\
            & \, \le C_4^{p^\prime} \norm{z}^{p^\prime}_{\Ls^p(J, \Vs)} \sup_{\abs{h} \le 2\veps} \sup_{\abs{\eta} \le 2 \veps} \paren{\int\limits_0^T \norm{u(t + h) - u(t + \eta)}^{p}_{\Vs} \dt}^{p^\prime/p},
    \end{align*}
    where $C_3 := C_2 \paren{8^{p}c^p C_p}^{1/(p-1)}$ and $C_4 := C_3^{1/p^\prime} T^{(p-2)/p}$. Analogously, there exists a constant $C_5 > 0$ independent of $\veps$ such that
    \[
        \abs{I_2 + I_4^a}^{p^\prime} \le C_5^{p^\prime} \norm{z}^{p^\prime}_{\Ls^p(J, \Vs)} \sup_{\abs{h} \le 2\veps} \sup_{\abs{\eta} \le 2 \veps} \paren{\int\limits_0^T \norm{u(t + h) - u(t + \eta)}^{p}_{\Vs} \dt}^{p^\prime/p}.
    \]
    The last term is clearly bounded by
    \[
        \dfrac{2\veps}{T} \abs{\inprod{\alpha \pa_t u, z}} \le \dfrac{2\veps}{T} \norm{\alpha \pa_t u}_{\Ls^{p^\prime}(J, \Vs^\prime)} \norm{z}_{\Ls^p(J, \Vs)}.
    \]
    Therefore, we end up with
    \begin{align*}
        & \norm{\alpha \pa_t u_\veps - \alpha \pa_t u}_{\Ls^{p^\prime}(J, \Vs^\prime)} \le C_1 \sup_{\abs{h} \leq 2 \veps} \paren{\int\limits_0^T \norm{(\alpha \pa_t u)(t + h) - (\alpha \pa_t u)(t)}_{\Vs^\prime}^{p^\prime} \dt}^{1/p^\prime}  \\
        & \q + \paren{C_4 + C_5} \sup_{\abs{h} \le 2\veps} \, \sup_{\abs{\lambda} \le 2\veps} \paren{\int\limits_0^T \norm{u(t + h) - u(t + \lambda)}^p_{\Vs} \dt}^{1/p} + \dfrac{2\veps}{T} \norm{\alpha \pa_t u}_{\Ls^{p^\prime}(J, \Vs^\prime)}.
    \end{align*}
    Taking into limit $\veps \to 0$ and invoking the mean continuity of $u$ and $\alpha \pa_t u$ as elements in $\Ls^p(J, \Vs)$ and $\Ls^{p^\prime}(J, \Vs^\prime)$, respectively, we arrive at \eqref{eq:conv_dualV}, and hence complete the proof.
\end{proof}

\begin{remark}
    In the case $\alpha = 1$, the distribution $\alpha \pa_t u$ defined by \eqref{eq:def} coincides with the weak time derivative $\pa_t u$ when $\vb \cdot \nv = 0$ on the boundary $\pa\Om$ (including the static case, i.e., $\vb = \zrb$ in $Q$), or $\left.u\right|_{\pa\Om} = 0$. These settings have been thoroughly investigated in \cite{Gross2011} and \cite{NLH+2025}, where the use of the standard Sobolev--Bochner space $\Ws^{1, 2, 2}(J, \Vs, \Vs^\prime)$ for the weak solution is justified. 
    In more general situations, especially when $\alpha$ has jump at the evolving interface $\Gm(t)$, the condition $\pa_t u \in \Ls^1(J, \Vs^\prime)$ does not imply the well-definedness of $\alpha \pa_t u$ in the same space or in any space of distributions that take values in $\Vs^\prime$. Additional terms arise in the definition \cref{eq:def} as a result of the Reynolds transport theorem \eqref{eq:RTT}. Without these terms, the defined $\alpha \pa_t u$ would not match the actual product of $\alpha$ and $\pa_t u$ when $u$ is sufficiently regular -- for instance, if $u$ satisfies $\pa_t u \in \Ls^1(J, \Ls^2(\Om))$. However, incorporating these terms requires the inclusion $\Vs \subseteq \Hs^1(\Om)$. Furthermore, the proof of the density of $\DD(J, \Vs)$ in the corresponding Sobolev--Bochner spaces $\Ws^{1, p, p^\prime}_{\alpha}(J, \Vs, \Vs^\prime)$ (cf. Lemma~\ref{lem:density}) relies on the continuous embeddings $\Vs \emb \Hs^1(\Om)$ and $\Ls^p(J, \Vs) \emb \Ls^{p^\prime}(J, \Vs)$ -- the latter of which fails when $p \in (1, 2).$
\end{remark}

\begin{remark}
    \Cref{def:derivative} of $\alpha \pa_t u$ is equivalent to the following definition of the material derivative $\alpha \dot{u} \in \LL(\DD(J), \Vs^\prime)$
    \[
        \inprod{(\alpha \dot{u})(\phi), w} = - \int\limits_0^T \int\limits_\Om \paren{ \alpha u \dot{(\phi w)} + u \phi w \pa_t \alpha + \phi u w \dive{\paren{\alpha \vb}}} \dx \dt,
    \]
    for any $\phi \in \DD(J)$ and $w \in \Vs$, where $\dot{(\phi w)} = w \pa_t \phi + \phi \vb \cdot \nabla w$. This definition coincides with the definition in \cite[Eq.~(2.5)]{VR2018} when $\alpha$ is constant on each subdomain and $\dive{\vb} = 0$ in $Q$. 
    
    \Cref{def:derivative} also generalizes the definition of $\alpha \pa_t u$ in \cite{LSV2024}. Indeed, let us assume that the assumptions of \cite[Definition~3.1]{LSV2024} are satisfied, i.e., $u \in \Ls^1(J, \Vs)$ and there exists $g \in \Ls^1(J, \Vs^\prime)$ such that
    \[
        \dfrac{\di}{\dt} \paren{\alpha(t) u(t), w} = \inprod{g(t), w},
    \]
    for any $w \in \Vs$ and for a.a. $t \in J$. In addition, we assume that the material coefficient $\alpha$ is piecewise-constant. Then, the definition \eqref{eq:def} turns out to be
    \[
        \inprod{(\alpha \pa_t u)(\phi), w} = \int\limits_0^T \phi(t) \paren{\inprod{g(t), w} - \int\limits_{\Om} \alpha(t) \dive{\paren{u(t) w \vb(t)}} \dx} \dt,
    \]
    which is valid for any $\phi \in \DD(J)$ and $w \in \Vs$. By a density argument, we can conclude that $\alpha \pa_t u \in \Ls^1(J, \Vs^\prime)$ and 
    \[
        \inprod{(\alpha \pa_t u)(t), w} = \inprod{g(t), w} - \int\limits_{\Om} \alpha(t) \dive{\paren{u(t) w \vb(t)}} \dx,
    \]
    for a.a. $t \in J$. This relation gives us back the definition of $\alpha \pa_t u$ in \cite[Definition~3.1]{LSV2024}.
\end{remark}

It is noteworthy that the space $\Ws^{1, p, p^\prime}_{\alpha}(J, \Vs, \Vs^\prime)$ is not embedded in $\Cs\paren{\overline{J}, \Ls^2(\Om)}$. Nevertheless, the following lemma establishes an analogous result that serves as a substitute for this lack of continuous embedding.

\begin{lemma}
\label{lem:IBP}
    Given $p \in [2, \infty)$. If $u \in \Ws^{1, p, p^\prime}_{\alpha}(J, \Vs, \Vs^\prime)$, then $\sqrt{\alpha} u \in \Cs\paren{\ovl{J}, \Ls^2(\Om)}$. In addition, there exists a constant $\wt{C}_{\textup{emb}, p} > 0$ such that 
    \[
        \norm{\sqrt{\alpha} u}_{\Cs\paren{\ovl{J}, \Ls^2(\Om)}} \le \wt{C}_{\textup{emb}, p} \norm{u}_{\Ws^{1, p, p^\prime}_{\alpha}(J, \Vs, \Vs^\prime)},
    \]
    for any $u \in \Ws^{1, p, p^\prime}_{\alpha}(J, \Vs, \Vs^\prime)$. Moreover, the following integration by parts formula holds for any $u, z \in \Ws^{1, p, p^\prime}_{\alpha}(J, \Vs, \Vs^\prime)$ and any $t_1, t_2 \in \ovl{J}$
    \begin{multline}
        \label{eq:IBP}
         \int\limits_{t_1}^{t_2} \inprod{(\alpha \pa_t u)(t), z(t)} + \inprod{(\alpha \pa_t z)(t), u(t)} \dt = \paren{(\sqrt{\alpha} u)(t_2), (\sqrt{\alpha} z)(t_2)} \\
         - \paren{(\sqrt{\alpha} u)(t_1), (\sqrt{\alpha} z)(t_1)} - \int\limits_{t_1}^{t_2} \int\limits_\Om \paren{u z \pa_t \alpha + \dive{\paren{\alpha u z \vb}}} \dx \dt.
    \end{multline}
\end{lemma}

\begin{proof}
    We note that \eqref{eq:IBP} holds for any $u, z \in \Cs^1\paren{\ovl{J}, \Vs}$ following the Reynolds transport theorem \eqref{eq:RTT}. For any $t \in \ovl{J}$ and $u \in \Cs^1\paren{\ovl{J}, \Vs}$, we can use the identity \eqref{eq:IBP} with $z = u, t_2 = t$, and $t_1 \in \ovl{J}$ satisfying $\norm{u(t_1)}^2_{\Ls^2(\Om)} = \tfrac{1}{T} \norm{u}^2_{\Ls^2(J, \Ls^2(\Om))}$ (such $t_1$ exists by the mean value theorem) to see that
    \begin{align*}
        & \norm{(\sqrt{\alpha} u)(t)}^2_{\Ls^2(\Om)} \\
        & \q = \norm{(\sqrt{\alpha} u)(t_1)}^2_{\Ls^2(\Om)} + \paren{\norm{(\sqrt{\alpha} u)(t)}^2_{\Ls^2(\Om)} - \norm{(\sqrt{\alpha} u)(t_1)}^2_{\Ls^2(\Om)}} \\
        & \q \le C_\alpha \norm{u(t_1)}^2_{\Ls^2(\Om)} + 2\abs{ \,\, \int\limits_{t_1}^{t} \inprod{\alpha \pa_s u, u} \ds} + \abs{\,\, \int\limits_{t_1}^{t} \int\limits_\Om \paren{u^2 \pa_s \alpha + \dive{\paren{\alpha u^2 \vb}}} \dx \ds} \\
        & \q \le C_{\alpha} T^{-1} \norm{u}^2_{\Ls^2(J, \Ls^2(\Om))} + 2 \norm{\alpha \pa_t u}_{\Ls^{p^\prime}(J, \Vs^\prime)} \norm{u}_{\Ls^p(J, \Vs)} + C_{\alpha} \paren{C_{\vb} + 1} \norm{u}^2_{\Ls^2(J, \Hs^1(\Om))}\\
        & \q \le \wt{C}^2_{\textup{emb}, p} \norm{u}^2_{\Ls^p(J, \Vs)} + \norm{\alpha \pa_t u}^2_{\Ls^{p^\prime}(J, \Vs^\prime)} \\
        & \q \le \wt{C}^2_{\textup{emb}, p} \norm{u}^2_{\Ws^{1, p, p^\prime}_{\alpha}(J, \Vs, \Vs^\prime)},
    \end{align*}
    where the constant
    \begin{equation}
    \label{eq:constant}
        \wt{C}_{emb, p} := \sqrt{ \paren{C_{\vb} + 1 + T^{-1}} C_{\alpha} C_{\text{emb}}^2 T^{(p-2)/p} + 1}.
    \end{equation}
    In the second last estimate, we have employed the continuous embeddings $\Vs \emb \Hs^1(\Om)$ and $\Ls^p(J, \Vs) \emb \Ls^2(J, \Vs)$, with norms $C_{emb}$ and $T^{(p-2)/2p}$, respectively. The proof is then concluded by invoking the density of $\Cs^1\paren{\ovl{J}, \Vs}$ in $\Ws^{1, p, p^\prime}_{\alpha}(J, \Vs, \Vs^\prime)$, as a consequence of Lemma~\ref{lem:density}.
\end{proof}

\begin{remark}
\label{rem:supp}
    In the space $\Ls^2(\Sm(t))$, the two norms $\norm{\cdot}$ and $\norm{\sqrt{\alpha(t)} \cdot}$ are equivalent, i.e.,
    \[
        \sqrt{\alpha_0} \norm{u}_{\Ls^2(\Sm(t))} \le \norm{\sqrt{\alpha(t)} u}_{\Ls^2(\Sm(t))} \le \sqrt{C_{\alpha}} \norm{u}_{\Ls^2(\Sm(t))}, \qqq \fa u \in \Ls^2(\Sm(t)).
    \]
    Lemma~\ref{lem:IBP} implies that if $u \in \Ws^{1, p, p^\prime}_{\alpha}(J, \Vs, \Vs^\prime)$, with $p \in [2, \infty)$, then $u(t) \in \Ls^2(\Sm(t))$ for all $t \in \ovl{J}$. In addition, it holds that
    \begin{equation}
    \label{eq:embedding}
        \max\limits_{t \in \ovl{J}} \norm{u(t)}_{\Ls^2(\Sm(t))} \le \dfrac{\wt{C}_{\textup{emb}, p}}{\sqrt{\alpha_0}} \norm{u}_{\Ws^{1, p, p^\prime}_{\alpha}(J, \Vs, \Vs^\prime)}.
    \end{equation}
\end{remark}

\section{Variational formulation}
\label{sec:varf}

This section is devoted to deriving the variational formulation for \eqref{eq:pde} using the function spaces $\Ws^{1, p, p^\prime}_{\alpha}(J, \Vs, \Vs^\prime)$ and establishing its well-posedness. Equation \eqref{eq:pde} must be complemented with an initial condition. Based on the arguments in Lemma~\ref{lem:IBP} and Remark~\ref{rem:supp}, the initial condition shall be given by
\begin{equation}
\label{eq:IC}
    u(0) = g_2 \qqqqqqqq \text{in} \q \Ls^2(\Sm_0),
\end{equation}
where $\Sm_0 := \Sm(0)$. We denote the Banach spaces
\[
    \XX := \Ws^{1, 2, 2}_{\alpha}(J, \Vs, \Vs^\prime), \qqqqq \YY := \Ls^2(J, \Vs) \times \Ls^2(\Sm_0),
\]
equipped with the norms
\[
    \norm{u}^2_{\XX} :=\norm{u}^2_{\Ls^2(J, \Vs)} + \norm{\alpha \pa_t u}^2_{\Ls^2(J, \Vs^\prime)}, \qqq \norm{y}_{\YY} := \norm{y_1}_{\Ls^2(J, \Vs)} + \norm{y_2}_{\Ls^2(\Sm_0)},
\]
where $y = \paren{y_1, y_2} \in \YY$.
Given $g_1 \in \Ls^2(J, \Vs^\prime), g_2 \in \Ls^2(\Sm_0)$, and $A(t) \in \LL(\Vs, \Vs^\prime)$ such that $J \ni s \mapsto \inprod{A(s) u, w}$ is Lebesgue measurable for all $u, w \in \Vs$, the space-time variational formulation of \eqref{eq:pde} and \eqref{eq:IC} reads: find $u \in \XX$ such that
\begin{equation}
\label{eq:vf}
    B(u, y) = f(y) \qqqqqq \fa y := (y_1, y_2) \in \YY,
\end{equation}
where the bilinear form $B : \XX \times \YY \to \R$ is defined by
\[
    B(u, y) := \int\limits_0^T \inprod{(\alpha \pa_t u)(t), y_1(t)} + \inprod{A(t) u(t), y_1(t)} \dt + \paren{u(0), y_2},
\]
and the linear form $f : \YY \to \R$ is given by
\[
    f(y) := \int\limits_0^T \inprod{g_1(t), y_1(t)} \dt + \paren{g_2, y_2}.
\]
The inner product in $\Ls^2(\Sm_0)$ is also denoted by $\paren{\cdot, \cdot}$. To establish the well-posedness of \eqref{eq:vf}, we make the following assumptions.
\begin{assumption}
\label{asmp:A}
    There exists a triple of real constants $\paren{c_A, C_A, \lambda_0}$ such that $0 < c_A \le C_A$ and for a.a. $t \in J$, the operator $A(t) \in \LL(\Vs, \Vs^\prime)$ satisfies
    \begin{enumerate}
        \item Boundedness: 
        \begin{equation}
        \label{eq:boundedness}
            \abs{\inprod{A(t) w, \psi}} \le C_A \norm{w}_{\Vs} \norm{\psi}_{\Vs}, \qqqqqqqqq \q \fa w, \psi \in \Vs.
        \end{equation}
        \item Coercivity:
        \begin{equation}
        \label{eq:coercivity}
            \inprod{A(t) w, w} + \paren{\lambda_0 - \lambda_1} \norm{\sqrt{\alpha(t)} w}^2_{\Ls^2(\Om)} \ge c_A \norm{w}^2_{\Vs}, \qq \fa w \in \Vs,
        \end{equation}
        with 
        \begin{equation}
        \label{eq:lambda1}
            \lambda_1 := \dfrac{C^2_{\alpha} C^2_{\textup{emb}} (C_{\vb} + 1)^2}{2 c_A \alpha_0}.
        \end{equation}
    \end{enumerate}
\end{assumption}

\begin{remark}
The conditions \eqref{eq:boundedness} and \eqref{eq:coercivity} are standard assumptions for establishing the well-posedness of the variational formulation of parabolic problems (see, e.g., \cite{SS2009,Andreev2013} for the classical case $\alpha = 1$ and $\vb = \zrb$). The coercivity assumption \eqref{eq:coercivity} indicates that $A(t)$ need not be $\Vs$-elliptic to ensure the well-posedness of \eqref{eq:vf}. The constant $\lambda_1$, which could be incorporated into $\lambda_0$, is introduced to handle additional terms arising due to the Reynolds transport theorem. 
\end{remark}

\begin{remark}
\label{rem:lambda}
    In the subsequent analysis of the well-posedness, we may set $\lambda_0 = 0$ in \eqref{eq:coercivity} without loss of generality. Indeed, following the proof of \cite[Theorem~5.1]{SS2009}, we write $u(t) = \hat{u}(t) e^{\lambda_0 t}, y_1(t) = \check{y}_1(t) e^{-\lambda_0 t}$ and $g_1(t) = \hat{g}_1(t) e^{\lambda_0 t}$. A straightforward calculation shows that $u$ solves \eqref{eq:vf} if and only if $\hat{u}$ solves
    \begin{equation}
    \label{eq:vf2}
        \wh{B}(\hat{u}, \check{y}) = \hat{f}(\check{y}) \qqqqqq \fa \check{y} = (\check{y}_1, y_2) \in \YY,
    \end{equation}
    where
    \[
        \wh{B}(\hat{u}, \check{y}) := \int\limits_0^T \inprod{(\alpha \pa_t \hat{u})(t), \check{y}_1(t)} + \inprod{\wh{A}(t) \hat{u}(t), \check{y}_1(t)} \dt + \paren{\hat{u}(0), y_2},
    \]
    and
    \[
        \hat{f}(\check{y}) = \int\limits_0^T \inprod{\hat{g}_1(t), \check{y}_1(t)} \dt + \paren{g_2, y_2}.
    \]
    The operator $\wh{A}(t) \in \LL(\Vs, \Vs^\prime)$ is defined for a.a. $t\in J$ by
    \[
        \inprod{\wh{A}(t) w, \psi} = \lambda_0 \paren{\alpha(t) w, \psi} + \inprod{A(t) w, \psi} \qqqqqq \fa w, \psi \in \Vs.
    \]
    One can easily verify that the operator $\wh{A}(t)$ satisfies \Cref{asmp:A} with the triplet $\paren{c_A, C_A + \abs{\lambda_0} C_{\alpha} C_{\text{emb}}^2, 0}$.
\end{remark}

Next, we prove the well-posedness of the variational problem \eqref{eq:vf} following the Banach--Ne\v{c}as--Babu\v{s}ka theorem, see \cite[Theorem~25.9]{Ern2021b}.

\begin{lemma}[Inf-sup condition]
\label{lem:infsup}
    There exists a constant $c_B > 0$ such that
    \[
        \sup_{y \in \YY \setminus \brac{0}} \dfrac{B(u, y)}{\norm{y}_{\YY}} \ge c_B \norm{u}_{\XX} \qqqqqqq \fa u \in \XX.
    \]
\end{lemma}

\begin{proof}
    As noted in \Cref{rem:lambda}, we can set $\lambda_0 = 0$ in \eqref{eq:coercivity} without loss of generality. We introduce the Riesz representation operator $R : \Ls^2(J, \Vs^\prime) \to \Ls^2(J, \Vs)$, i.e., 
    \[
        (Rz, y_1)_{\Ls^2(J, \Vs)} = \inprod{z, y_1} \qqqqqqq \fa y_1 \in \Ls^2(J, \Vs),
    \]
    where $\paren{\cdot, \cdot}_{\Ls^2(J, \Vs)}$ denotes the inner product in $\Ls^2(J, \Vs)$. It is clear that
    \[
        \norm{R z}_{\Ls^2(J, \Vs)} = \norm{z}_{\Ls^2(J, \Vs^\prime)} \qqqqqqq \fa z \in \Ls^2(J, \Vs^\prime).
    \]
    For any $u \in \XX$, we first take $\ovl{y} := \paren{R(\alpha \pa_t u), 0} \in \YY$ to have that
    \begin{align*}
        B\paren{u, \ovl{y}} 
        & = \norm{R(\alpha \pa_t u)}^2_{\Ls^2(J, \Vs)} + \int\limits_0^T \inprod{A(t) u(t), R(\alpha \pa_t u)(t)} \dt \\
        & \ge \norm{R(\alpha \pa_t u)}^2_{\Ls^2(J, \Vs)} - C_A \norm{u}_{\Ls^2(J, \Vs)} \norm{R(\alpha \pa_t u)}_{\Ls^2(J, \Vs)} \\
        & \ge (1 - \veps) \norm{\alpha \pa_t u}^2_{\Ls^2(J, \Vs^\prime)} - \dfrac{C^2_A}{4 \veps} \norm{u}^2_{\Ls^2(J, \Vs)},
    \end{align*}
    for any $\veps > 0$. Next, choosing $\tilde{y} := \paren{2 u,  \alpha(0) u(0)} \in \YY$ and applying the integration by parts formula \eqref{eq:IBP} give us that
    \begin{align*}
        B\paren{u, \tilde{y}} & = 2 \int\limits_0^T \inprod{(\alpha \pa_t u)(t), u(t)} + \inprod{A(t) u(t), u(t)} \dt + \norm{\paren{\sqrt{\alpha} u}(0)}^2_{\Ls^2(\Om)} \\
        & = \norm{\paren{\sqrt{\alpha} u}(T)}^2_{\Ls^2(\Om)} + \int\limits_0^T 2 \inprod{A(t) u(t), u(t)} - \int\limits_\Om \paren{ u^2 \pa_t \alpha + \dive{\paren{\alpha u^2 \vb}}} \dx \dt \\
        & \ge \norm{\paren{\sqrt{\alpha} u}(T)}^2_{\Ls^2(\Om)} + c_A \norm{u}^2_{\Ls^2(J, \Vs)},
    \end{align*}
    where we have employed the following estimate
    \begin{align*}
        \int\limits_0^T \int\limits_\Om \paren{u^2 \pa_t \alpha + \dive{\paren{\alpha  u^2 \vb}}} \dx \dt & \le 2 C_{\alpha} \paren{C_{\vb} + 1} \int\limits_0^T \norm{u(t)}_{\Ls^2(\Sm(t))} \norm{u(t)}_{\Hs^1(\Om)} \dt \\
        & \le \dfrac{2 C_{\alpha} C_{\text{emb}} (C_{\vb} + 1)}{\sqrt{\alpha_0}} \norm{\sqrt{\alpha} u}_{\Ls^2(J, \Ls^2(\Om))} \norm{u}_{\Ls^2(J, \Vs)} \\
        & \le c_A \norm{u}^2_{\Ls^2(J, \Vs)} + 2\lambda_1
        \norm{\sqrt{\alpha} u}^2_{\Ls^2(J, \Ls^2(\Om))},
    \end{align*}
    with $\lambda_1$ defined in \eqref{eq:lambda1}. 
    Now, we can take $y = \ovl{y} + \beta \tilde{y} \in \YY$, with $\beta > 0$, to arrive at
    \begin{align*}
        B(u, y) & \ge (1 - \veps) \norm{\alpha \pa_t u}^2_{\Ls^2(J, \Vs^\prime)} + \paren{\beta c_A - \dfrac{C_A^2}{4\veps}} \norm{u}^2_{\Ls^2(J, \Vs)} \\
        & \ge \min \brac{1 - \veps, \beta c_A - \dfrac{C_A^2}{4\veps}} \norm{u}^2_{\XX}.
    \end{align*}
    On the other hand, we can bound $y$ as follows
    \begin{align*}
        \norm{y}_{\YY} 
        & \le \norm{R(\alpha \pa_t u)}_{\Ls^2(J, \Vs)} + 2\beta \norm{u}_{\Ls^2(J, \Vs)} + \beta C_{\alpha} \norm{u(0)}_{\Ls^2(\Sm_0)} \\
        & \le  \paren{1 +  2 \beta+ \dfrac{\beta C_{\alpha} \wt{C}_{\textup{emb}, 2}}{\sqrt{\alpha_0}}} \norm{u}_{\XX},
    \end{align*}
    where the inequality \eqref{eq:embedding} has been used with the constant $\wt{C}_{\textup{emb}, 2}$ defined in \eqref{eq:constant}. 
    Finally, fixing $0 < \veps < 1$ and then choosing a sufficiently large $\beta > 0$ such that $4 \veps \beta c_A > C_A^2$, we complete the proof.
\end{proof}

\begin{lemma}[Surjectivity]
    If $B(u, y) = 0$ for all $u \in \XX$, then $y = 0$ in $\YY$.
\end{lemma}

\begin{proof}
    Let $y \in \YY$ such that for all $u \in \XX$
    \begin{equation}
    \label{eq:surjectivity}
        B(u, y) = \int\limits_0^T \inprod{(\alpha \pa_t u)(t), y_1(t)} + \inprod{A(t) u(t), y_1(t)} \dt + \paren{u(0), y_2} = 0.
    \end{equation}
    We first take $u = \phi w \in \XX$, where $\phi \in \DD(J)$ and $w \in \Vs$, to have that
    \[
        B(u, y) = \int\limits_0^T \int\limits_{\Om} \alpha w y_1 \pa_t \phi \dx \dt + \int\limits_0^T \inprod{A(t) u(t), y_1(t)} \dt = 0.
    \]
    \Cref{def:derivative} implies that $\alpha \pa_t y_1 \in \LL(\DD(J), \Vs^\prime)$ satisfies
    \[
        \inprod{(\alpha \pa_t y_1)(\phi), w} = \int\limits_0^T \inprod{A(t)^\prime y_1(t), u(t)} \dt - \int\limits_0^T \int\limits_\Om \paren{u y_1 \pa_t \alpha + \dive{\paren{\alpha u y_1 \vb}}} \dx \dt,
    \]
    for any $u = \phi w \in \DD(J) \otimes \Vs$, where $A(t)^\prime \in \LL(\Vs, \Vs^\prime)$ is the dual operator to $A(t)$ for a.a. $t \in J$. Due to the density of $\DD(J) \otimes \Vs$ in $\Ls^2(J, \Vs)$ (see \cite{SS2009}), the right-hand side defines a bounded linear functional (of $u$) on $\Ls^2(J, \Vs)$. Therefore, we can conclude $\alpha \pa_t y_1 \in \Ls^2(J, \Vs^\prime)$ or $y_1 \in \XX$. In addition, the following holds for any $u \in \Ls^2(J, \Vs)$
    \begin{multline}
    \label{eq:dual}
        \int\limits_0^T \inprod{(\alpha \pa_t y_1)(t), u(t)} - \inprod{A(t)u(t), y_1(t)} \dt + \int\limits_0^T \int\limits_\Om \paren{u y_1 \pa_t \alpha + \dive{\paren{\alpha u y_1 \vb}}} \dx \dt = 0.
    \end{multline}
    Next, choosing $u = \phi \tilde{u} \in \XX$ in \eqref{eq:surjectivity} and \eqref{eq:dual}, with $\phi \in \Cs^1(\ovl{J})$ such that $\phi(0) = 1$ and $\phi(T) = 0$ and $\tilde{u} \in \XX$, then adding them together and employing the integration by parts formula \eqref{eq:IBP}, we arrive at
    \[
        \paren{\paren{\sqrt{\alpha} \tilde{u}}(0), \paren{\sqrt{\alpha} y_1} (0)} = \paren{\tilde{u}(0), y_2},
    \]
    which holds true for any $\tilde{u} \in \XX$. Taking into account the density of $\Vs$ in $\Ls^2(\Om)$, we can conclude that $y_2 = \alpha(0) y_1 (0)$ in $\Ls^2(\Sm_0)$. Finally, taking $u = y_1 \in \XX$ in \eqref{eq:surjectivity} and performing analogous estimates as in the proof of \Cref{lem:infsup}, we end up with 
    \begin{align*}
        0 =  B(y_1, y) & = \int\limits_0^T \inprod{(\alpha \pa_t y_1)(t), y_1(t)} + \inprod{A(t) y_1(t), y_1(t)} \dt + \norm{\paren{\sqrt{\alpha} y_1}(0)}^2_{\Ls^2(\Om)}\\
        & \ge \dfrac{1}{2} \norm{\paren{\sqrt{\alpha} y_1}(T)}^2_{\Ls^2(\Om)} + \dfrac{1}{2} \norm{\paren{\sqrt{\alpha} y_1}(0)}^2_{\Ls^2(\Om)} + \dfrac{c_A}{2} \norm{y_1}^2_{\Ls^2(J, \Vs)},
    \end{align*}
    implying that $y_1 = 0$ in $\Ls^2(J, \Vs)$ and $y_2 = 0$ in $\Ls^2(\Sm_0)$, or equivalently $y = 0$ in $\YY$. We have completed the proof.
\end{proof}

As a consequence of the Banach--Ne\v{c}as--Babu\v{s}ka theorem, we conclude the paper with the following theorem. 

\begin{theorem}[Well-posedness]
    The variational problem \eqref{eq:vf} has a unique solution $u \in \XX$ for any $g_1 \in \Ls^2(J, \Vs^\prime)$ and $g_2 \in \Ls^2(\Sm_0)$. Moreover, the following a priori estimate holds
    \[
        \norm{u}_{\XX} \le \dfrac{1}{c_B} \paren{\norm{g_1}_{\Ls^2(J, \Vs^\prime)} + \norm{g_2}_{\Ls^2(\Sm_0)}}.
    \]
\end{theorem}

\section{Conclusions}
\label{sec:conclusions}

In this paper, we have developed a functional framework for the variational approach of parabolic PDEs on a fixed domain comprising evolving subdomains and interfaces. The function spaces introduced naturally generalize standard Sobolev--Bochner spaces, accommodating a material coefficient that may be discontinuous across the evolving interface and partially degenerate in subdomains. We have established key properties of these spaces, including the density of smooth functions, an appropriate ``embedding'' theory, and an integration by parts formula, which are essential in deriving the variational formulation. Finally, we have demonstrated the well-posedness of the space-time variational problem in this natural setting using standard arguments.

The result presented in this work provides a foundation for developing numerical methods for parabolic PDEs with evolving subdomains. Potential directions include classical time-stepping (Rothe's method) finite element schemes \cite{LSV2022a,LSV2022c}, high-order time-stepping and finite element methods \cite{FR2017,FJR2023}, discontinuous Galerkin approaches \cite{Moore2019}, and space–time finite element methods \cite{Steinbach2015,NLH+2025}. A detailed investigation of these numerical schemes will be the subject of future work.

\section*{Funding}

This work was funded by the European Research Council (ERC) under the European Union’s Horizon 2020 Research and innovation programme (Grant agreement No. 101001847) and the FWO Senior Research Grant G083525N.

\bibliography{abrv_ref}
\bibliographystyle{unsrt}
\end{document}

%% file: lvc_style.tex
\newcommand{\R}{\mathbb R}

\newcommand{\DD}{\mathcal D}

\newcommand{\LL}{\mathcal L}

\newcommand{\XX}{\mathcal X}
\newcommand{\YY}{\mathcal Y}

\DeclareMathOperator{\Ls}{L}

\DeclareMathOperator{\Hs}{H}

\DeclareMathOperator{\Ws}{W}

\DeclareMathOperator{\Cs}{C}

\DeclareMathOperator{\Vs}{V}

\DeclareMathOperator{\Xs}{X}

\DeclareMathOperator{\supp}{supp}

\newcommand{\brac}[1]{\left\lbrace{#1}\right\rbrace}

\newcommand{\paren}[1]{\left({#1}\right)}
\newcommand{\norm}[1]{\left\lVert{#1}\right\rVert}

\newcommand{\abs}[1]{\left\vert{#1}\right\vert}

\newcommand{\inprod}[1]{\left\langle{#1}\right\rangle}

\newcommand{\wt}[1]{\widetilde{#1}}
\newcommand{\wh}[1]{\widehat{#1}}
\newcommand{\ovl}[1]{\overline{#1}}

\newcommand{\bsm}[1]{\boldsymbol{#1}}
\newcommand{\dive}[1]{\nabla\cdot{#1}}

\newcommand{\veps}{\varepsilon}

\newcommand{\pa}{\partial}

\newcommand{\Gm}{\Gamma}

\newcommand{\Om}{\Omega}

\newcommand{\Sm}{\Sigma}

\newcommand{\fa}{\forall}
\newcommand{\sst}{\subset}

\newcommand{\xb}{\bsm{x}}

\newcommand{\vb}{\bsm{v}}

\newcommand{\nv}{\textbf{n}}

\newcommand{\zrb}{\bsm{0}}

\newcommand{\emb}{\hookrightarrow}

\newcommand{\dx}{\,\mathrm{d}\xb}

\newcommand{\ds}{\,\mathrm{d}s}
\newcommand{\dt}{\,\mathrm{d}t}
\newcommand{\di}{\,\mathrm{d}}

\newcommand{\q}{\quad}
\newcommand{\qq}{\qquad}
\newcommand{\qqq}{\qquad\quad}
\newcommand{\qqqq}{\qquad\qquad}
\newcommand{\qqqqq}{\qquad\qquad\quad}
\newcommand{\qqqqqq}{\qquad\qquad\qquad}
\newcommand{\qqqqqqq}{\qquad\qquad\qquad\quad}
\newcommand{\qqqqqqqq}{\qquad\qquad\qquad\qquad}
\newcommand{\qqqqqqqqq}{\qquad\qquad\qquad\qquad\quad}